\documentclass[12pt]{amsart}
\usepackage{amssymb,amsthm,amsmath,epsfig,latexsym}
\usepackage{calc}
\usepackage{times}
\begin{document}

\newcommand{\mmbox}[1]{\mbox{${#1}$}}
\newcommand{\proj}[1]{\mmbox{{\mathbb P}^{#1}}}
\newcommand{\Cr}{C^r(\Delta)}
\newcommand{\CR}{C^r(\hat\Delta)}
\newcommand{\affine}[1]{\mmbox{{\mathbb A}^{#1}}}
\newcommand{\Ann}[1]{\mmbox{{\rm Ann}({#1})}}
\newcommand{\caps}[3]{\mmbox{{#1}_{#2} \cap \ldots \cap {#1}_{#3}}}
\newcommand{\N}{{\mathbb N}}
\newcommand{\Z}{{\mathbb Z}}
\newcommand{\R}{{\mathbb R}}
\newcommand{\Tor}{\mathop{\rm Tor}\nolimits}
\newcommand{\Ext}{\mathop{\rm Ext}\nolimits}
\newcommand{\Hom}{\mathop{\rm Hom}\nolimits}
\newcommand{\im}{\mathop{\rm Im}\nolimits}
\newcommand{\rank}{\mathop{\rm rank}\nolimits}
\newcommand{\supp}{\mathop{\rm supp}\nolimits}
\newcommand{\arrow}[1]{\stackrel{#1}{\longrightarrow}}
\newcommand{\CB}{Cayley-Bacharach}
\newcommand{\coker}{\mathop{\rm coker}\nolimits}
\sloppy
\newtheorem{defn0}{Definition}[section]
\newtheorem{prop0}[defn0]{Proposition}
\newtheorem{conj0}[defn0]{Conjecture}
\newtheorem{thm0}[defn0]{Theorem}
\newtheorem{lem0}[defn0]{Lemma}
\newtheorem{corollary0}[defn0]{Corollary}
\newtheorem{example0}[defn0]{Example}

\newenvironment{defn}{\begin{defn0}}{\end{defn0}}
\newenvironment{prop}{\begin{prop0}}{\end{prop0}}
\newenvironment{conj}{\begin{conj0}}{\end{conj0}}
\newenvironment{thm}{\begin{thm0}}{\end{thm0}}
\newenvironment{lem}{\begin{lem0}}{\end{lem0}}
\newenvironment{cor}{\begin{corollary0}}{\end{corollary0}}
\newenvironment{exm}{\begin{example0}\rm}{\end{example0}}

\newcommand{\defref}[1]{Definition~\ref{#1}}
\newcommand{\propref}[1]{Proposition~\ref{#1}}
\newcommand{\thmref}[1]{Theorem~\ref{#1}}
\newcommand{\lemref}[1]{Lemma~\ref{#1}}
\newcommand{\corref}[1]{Corollary~\ref{#1}}
\newcommand{\exref}[1]{Example~\ref{#1}}
\newcommand{\secref}[1]{Section~\ref{#1}}

\newcommand{\std}{Gr\"{o}bner}
\newcommand{\jq}{J_{Q}}



\title {The minimum distance of sets of points and the minimum socle degree}

\author{\c Stefan O. Toh\v aneanu}
\address{Department of Mathematics\\ The University of Western Ontario\\ London, Ontario N6A 5B7\\}
\email{stohanea@uwo.ca}


\begin{abstract}
\noindent Let $\mathbb K$ be a field of characteristic 0. Let $\Gamma\subset\mathbb P^n_{\mathbb K}$ be a reduced finite set of points, not all contained in a hyperplane. Let $hyp(\Gamma)$ be the maximum number of points of $\Gamma$ contained in any hyperplane, and let $d(\Gamma)=|\Gamma|-hyp(\Gamma)$. If $I\subset R=\mathbb K[x_0,\ldots,x_n]$ is the ideal of $\Gamma$, then in \cite{t1} it is shown that for $n=2,3$, $d(\Gamma)$ has a lower bound expressed in terms of some shift in the graded minimal free resolution of $R/I$. In these notes we show that this behavior is true in general, for any $n\geq 2$: $d(\Gamma)\geq A_n$, where $A_n=\min\{a_i-n\}$ and $\oplus_i R(-a_i)$ is the last module in the graded minimal free resolution of $R/I$. In the end we also prove that this bound is sharp for a whole class of examples due to Juan Migliore (\cite{m}).
\end{abstract}
\maketitle

\section{Introduction}

Let $\mathbb K$ be a field of characteristic zero and let $\Gamma=\{P_1,\ldots,P_m\}\subset \mathbb P^n_{\mathbb K}$ be a reduced finite set of points, not all in a hyperplane (i.e., non-degenerate). Let $hyp(\Gamma)$ be the maximum number of points of $\Gamma$ lying in any hyperplane. Define \textit{the minimum distance of the set $\Gamma$} to be the number $$d(\Gamma)=m-hyp(\Gamma).$$ The reason we borrowed this terminology from coding theory is that $d(\Gamma)$ is exactly the minimum distance of the (equivalence class of) linear codes with generating matrix having as columns the coordinates of the points of $\Gamma$ (see \cite{tvn} for more details).

Denote with $R=\mathbb K[x_0,\ldots,x_n]$ the (homogeneous) ring of polynomials with coefficients in $\mathbb K$. Let $I\subset R$ be the ideal of $\Gamma$. The goal of these notes is to study $d(\Gamma)$ using the graded minimal free resolution of $R/I$.

Some preliminary results were obtained in \cite{gls} when $\Gamma$ is a complete intersection, and generalized in \cite{t1} when $\Gamma$ is (arithmetically) Gorenstein. In both situations $$d(\Gamma)\geq reg(R/I),$$ the Castelnuovo-Mumford regularity. The question became if this lower bound is true for any reduced non-degenerate finite set of points (\cite{t2}). As we will see below (Example \ref{exm:exampleone}), the answer is negative, yet we will still be able to give a lower bound for $d(\Gamma)$ in this general setup, in terms of the shifts in the graded minimal free resolution of $R/I$.

If $A=\displaystyle\oplus_{i=0}A_i$ is a graded Artinian $\mathbb K-$algebra with maximal ideal $\underline{m}=\displaystyle\oplus_{i>0}A_i$, then $soc(A)=0:\underline{m}$ is a finite dimensional graded $\mathbb K-$vector space, called \textit{the socle of $A$}. So $$soc(A)=\oplus\mathbb K(-b_i),$$ and the positive integers $b_i$ are called \textit{the socle degrees of $A$}.

In our case, if $\bar{I}$ is the Artinian reduction of $I$, the ideal of $\Gamma$, and if $$0\rightarrow F_n=\oplus R(-a_i)\rightarrow\cdots\rightarrow F_1\rightarrow R\rightarrow R/I\rightarrow 0$$ is the graded minimal free resolution of $R/I$, then the last module in the free resolution of $A=R/\bar{I}$ is $F_n(-1)=\oplus R(-(a_i+1))$ and sits in position $n+1$. So, by \cite{ku}, Lemma 1.3, the socle degrees of $A$ are exactly $$b_i=(a_i+1)-(n+1)=a_i-n.$$ We'll abuse the terminology by saying that the socle degrees of $A=R/\bar{I}$ are the socle degrees of $R/I$.

Denote $$A_n=\min\{a_i-n\}$$ to be the minimum value of the socle degrees.

In \cite{t1}, Theorem 4.1, we showed that if $\Gamma$ is any reduced non-degenerate finite set of points in $\mathbb P^k, k=2,3,$ then $d(\Gamma)\geq A_k$. In the first part of these notes we generalize this result (Theorem \ref{thm:main}) showing that if $\Gamma$ is any reduced non-degenerate finite set of points in $\mathbb P^n, n\geq 2,$ then $$d(\Gamma)\geq A_n,$$ and in the second part we investigate if this bound is sharp.

\section{A lower bound on the minimum distance of sets of points}

Let $\Gamma=\{P_1,\ldots,P_m\}\subset \mathbb P^n$ be a reduced non-degenerate finite set of points. We denoted with $hyp(\Gamma)$ the maximum number of points of $\Gamma$ contained in any hyperplane. To obtain the maximum number of points of $\Gamma$ contained in any hypersurface of degree $a$, by \cite{mp}, one should compute $hyp(v_a(\Gamma))$, where $v_a$ is the Veronese embedding of degree $a$ of $\mathbb P^n$ into $\mathbb P^{N_a}$, where $N_a= {{n+a}\choose{a}}-1$. Let us denote $$d(\Gamma)_a=|\Gamma|-hyp(v_a(\Gamma)).$$ Observe that $d(\Gamma)_1=d(\Gamma)$.

From \cite{t2} (using \cite{h}), $d(\Gamma)_a$ is the minimum distance of the evaluation code of order $a$ associated to $\Gamma$. With this fact in mind, \cite{t1}, Proposition 2.1, will constitute the key tool to prove our main result:

\begin{lem}\label{lem:recursion}(\cite{t1}) If $d(\Gamma)_b\geq 2$ for some $b\geq 2$, then for all $1\leq a\leq b-1$, we have $d(\Gamma)_a\geq d(\Gamma)_{a+1}+1$. Therefore, if $d(\Gamma)_b\geq 2$ for some $b\geq 1$, we have $d(\Gamma)_a\geq b-a+2$ for all $1\leq a\leq b$.
\end{lem}

In general, if $a\leq b$ then $d(\Gamma)_a\geq d(\Gamma)_b$.

\vskip .2in

Let $\Gamma'=\Gamma\setminus\{P_m\}$. Let $I=I(\Gamma)$ and $I'=I(\Gamma')$ be the homogeneous ideals in $R=\mathbb K[x_0,x_1,\ldots,x_n]$ of the sets $\Gamma$ and $\Gamma'$.

Since $\Gamma'\varsubsetneq\Gamma$, then $I\varsubsetneq I'$, and consider $$\delta(P_m)=\min\{d|\dim(I'_d)> \dim(I_d)\}\geq 1.$$ An element in $I'\setminus I$ is called a \textit{separator} of $P_m$, and $\delta(P_m)$ is called the \textit{degree of the point $P_m$ in $\Gamma$}. By \cite{gmr}, the Hilbert function of the $R/I$ and the degree of a point in $\Gamma$ are related by the following formula:

\begin{lem} \label{lem:HF}(\cite{gmr})
$$HF(R/I,i)=\left\{
                                        \begin{array}{ll}
                                          HF(R/I',i), & \hbox{if $0\leq i\leq \delta(P_m)-1$;} \\
                                          HF(R/I',i)+1, & \hbox{if $i\geq\delta(P_m)$.}
                                        \end{array}
                                      \right.$$
\end{lem}

Suppose the graded minimal free resolution of the $R-$ module $R/I$ is $$0\rightarrow F_n=\oplus R(-a_i)\rightarrow\cdots\rightarrow F_1\rightarrow R\rightarrow R/I\rightarrow 0,$$ and let $A_n=\min\{a_i-n\}$ be the minimum socle degree of $R/I$.

It was shown in \cite{abm}, for the case of points in $\mathbb P^2$, and, in general, in \cite{b} (using \cite{k}), for the case of points in $\mathbb P^n, n\geq 2$, that the degree of a point in $\Gamma$ is among the socle degrees of $R/I$.

\begin{lem} (\cite{b})\label{lem:separator} If $P$ is any point in $\Gamma$ and $\delta(P)$ is as above, then $$\delta(P)\geq A_n.$$
\end{lem}

Once we have this, we can prove the main result.

\begin{thm} \label{thm:main} In the above notations, $$d(\Gamma)\geq A_n.$$
\end{thm}

\begin{proof} The set $\Gamma$ is non-degenerate, so $A_n\geq 1$. If $A_n=1$, then the result is immediate since $d(\Gamma)\geq 1$ all the time. Assume that $A_n\geq 2$.

Let $$\delta=\delta(\Gamma)=\min\{\delta(P_i)|i=1,\ldots,m\}.$$ If $\delta=1$, then from Lemma \ref{lem:separator} $A_n=1$. So let us assume that $\delta\geq 2$ and consider $d(\Gamma)_{\delta-1}$.

By \cite{h}, for any $a\geq 1$, we have that $$d(\Gamma)_a=|\Gamma|-\max_{\Gamma'\subset\Gamma}\{|\Gamma'|: \dim(I(\Gamma')_a)>\dim(I(\Gamma)_a)\}.$$ So, if $d(\Gamma)_{\delta-1}=1$, then there exists $Q\in\Gamma$ such that $\dim(J_{\delta-1})>\dim(I_{\delta-1})$, where $J$ is the ideal of $\Gamma\setminus\{Q\}$. From Lemma \ref{lem:HF}, $$\delta-1\geq\delta(Q).$$ But this contradicts the minimality of $\delta$. Therefore, $$d(\Gamma)_{\delta-1}\geq 2.$$

From Lemma \ref{lem:separator} we have that $\delta-1\geq A_n-1$ and therefore, $$d(\Gamma)_{A_n-1}\geq d(\Gamma)_{\delta-1}\geq 2.$$ By using Lemma \ref{lem:recursion} with $b=A_n-1$ and $a=1$, we obtain $$d(\Gamma)=d(\Gamma)_1\geq (A_n-1)-1+2=A_n.$$ \end{proof}

\vskip .2in

\begin{exm} \label{exm:exampleone} Consider $\Gamma = \{[0,0,1],[0,1,0],[0,2,1],[0,3,1],[1,0,0]\}\subset \mathbb P^2$. The first four points lie on the line of equation $x=0$, and the fifth does not. Therefore $hyp(\Gamma)=4$ and $d(\Gamma)=5-4=1$. The ideal of $\Gamma$ in $R=\mathbb K[x,y,z]$ is $$I=\langle x,y\rangle\cap\langle x,z\rangle\cap \langle x,2z-y\rangle\cap\langle x,3z-y\rangle\cap\langle y,z\rangle.$$

With the help of Macaulay 2 by Grayson and Stillman, the minimal graded free resolution of $R/I$ is: $$0\rightarrow R(-5)\oplus R(-3)\rightarrow R(-4)\oplus R^2(-2)\rightarrow R\rightarrow R/I\rightarrow 0.$$ So $reg(R/I)=5-2=3$ and $A_2=3-2=1$.
\end{exm}

\section{Sets of points with minimum distance equal to $A_n$}

Example \ref{exm:exampleone} belongs to the class of examples for which $d(\Gamma)=A_n$. In this section we are going to investigate the following question: for given $n$ and $m$, under what conditions we can find, if it exists, a non-degenerate reduced finite set of $m$ points $\Gamma\subset\mathbb P^n$ with $d(\Gamma)=A_n$? Also we can ask a bit more: for given $n$, $m$ and $d(\Gamma)$, can we construct a non-degenerate reduced finite set of $m$ points $\Gamma\subset\mathbb P^n$ with $d(\Gamma)=A_n$?

Denote with $a(\Gamma)=\min\{a_i\}$ (we keep the same notations as before: $F_n=\oplus R(-a_i)$ is the last module in the graded minimal free resolution of $R/I$). Therefore, $A_n=a(\Gamma)-n$.

First of all, since $R(-a(\Gamma))$ is a direct summand in $F_n$, then $a(\Gamma)\geq n$. If $a(\Gamma)=n$, then one will have $R(-1)$ as a direct summand in $F_1$, which means that $I$ has a minimal generator of degree 1. This means that $\Gamma$ lies in a hyperplane and, therefore, $\Gamma$ is degenerate. So we must have that $$a(\Gamma)\geq n+1.$$

Let's see some simple cases:

\vskip .1in

\begin{exm}\label{exm:exampletwo} The case: $a(\Gamma)=n+1$. This is the case of Example \ref{exm:exampleone}. Construct $\Gamma$ as $m-1$ points lying in a hyperplane and one point outside this hyperplane. From Theorem \ref{thm:main}, since $d(\Gamma)=m-(m-1)=1$, we have $a(\Gamma)-n\leq 1$ and from the restriction above, we have $a(\Gamma)=n+1$. So this set satisfies the requirement $d(\Gamma)=A_n$.\end{exm}

\vskip .1in

\begin{exm} \label{exm:examplethree} The case $a(\Gamma)=n+2$. Since any $n$ points in $\mathbb P^n$ lie in a hyperplane, then $m\geq n+2$ (if $m=n+1$ we'd be in the case above). If $m=n+2$, let's pick $\Gamma$ to be a generic set of $n+2$ points in $\mathbb P^n$. By \cite{go}, $R/I$ is Gorenstein of regularity $r=2$. So $A_n=r=2$. Since $hyp(\Gamma)=n$, we have that $d(\Gamma)=(n+2)-n=2=A_n$.\end{exm}

\vskip .3in

In general, let us consider the following set of points $\Gamma$ in $\mathbb P^n$, suggested by Juan Migliore (\cite{m}).

Let $\Gamma_1\subset \mathbb P^n$ be a generic set of $\alpha$ points in $\mathbb P^{n-1}$ embedded in $\mathbb P^n$ (assume the hyperplane where they lie has equation $x_0=0$).

Let $\Gamma_2\subset \mathbb P^n$ be a set of $\beta$ distinct points on a line in $\mathbb P^n$ not contained in the above hyperplane. Assume that the coordinates of these points are $[1,u_i,1\ldots,1], 1\leq i\leq\beta, u_i\neq u_j$.

Let $\Gamma=\Gamma_1\cup\Gamma_2$ and we would like to have that $hyp(\Gamma)=\alpha$ (so one immediate restriction is that $\alpha\geq \beta+n-2$).

The goal is to see under what conditions $$d(\Gamma)=(\alpha+\beta)-\alpha=\beta=A_n.$$

\vskip .1in

Let $I,I_1,I_2\subset R=\mathbb K[x_0,\ldots,x_n]$ be the ideals of the sets $\Gamma, \Gamma_1$ and, respectively, $\Gamma_2$.

We have that $$I_2=\langle \prod_{i=1}^{\beta}(u_ix_0-x_1),x_2-x_0,\ldots,x_n-x_0\rangle$$ and $$I_1=\langle x_0,J\rangle,$$ where $J\subset S=\mathbb K[x_1,\ldots,x_n]$ is the ideal of the generic set of $\alpha$ points in $\mathbb P^{n-1}$.

First, let $s$ be the smallest integer such that $\alpha<{{s+n-1}\choose{n-1}}$. Since $J$ is the ideal of a generic set of $\alpha$ points in $\mathbb P^{n-1}$, then the Hilbert function is as nice as possible (in fact this is the definition of a generic set of points): $$HF(S/J,i)=\left\{
                         \begin{array}{ll}
                           {{i+n-1}\choose{n-1}}, & \hbox{if $i\leq s-1$;} \\
                           \alpha, & \hbox{if $i\geq s$.}
                         \end{array}
                       \right.$$

Suppose the minimal free resolution of $S/J$ is $$0\rightarrow C_{n-1}\rightarrow \cdots\rightarrow C_1\rightarrow S\rightarrow S/J\rightarrow 0.$$ Suppose that $u$ is the minimum shift in $C_{n-1}$. Then $u-(n-1)\geq s$; otherwise, moving down on the resolution to $C_1$ we'd have an element of degree $< s$ and this contradicts the Hilbert function. Also the Hilbert function tells us that the regularity of $S/J$ is $s$. So $S/J$ is \textit{level}: $$C_{n-1}=S^k(-(s+n-1)).$$

$J$ is minimally generated in degree $\geq s$ and the regularity of $S/J$ is $s$, therefore $$C_1=S^{p_1}(-s)\oplus S^{p_2}(-(s+1)).$$

Since $I_1=\langle x_0,J\rangle$, then the minimal free resolution of $R/I_1$ is: $$\mathbb G_{*}: 0\rightarrow G_n=C_{n-1}[x_0](-1)\rightarrow G_{n-1}=C_{n-2}[x_0](-1)\oplus C_{n-1}[x_0]\rightarrow\cdots$$ $$\rightarrow G_1=R(-1)\oplus C_1[x_0]\rightarrow R\rightarrow R/I_1\rightarrow 0,$$ where if $C_i=\oplus S(-c_{ij})$, we denoted $C_i[x_0]=\oplus R(-c_{ij})$.

Also, since $J$ is the ideal of points not all lying in a hyperplane, then $J\nsubseteq\langle x_2,\ldots,x_n\rangle$, and therefore one can assume that $$\langle J,x_2,\ldots,x_n\rangle=\langle x_1^v,x_2,\ldots,x_n\rangle,$$ for $v=s$ or $v=s+1$.

\vskip .1in

We have that $I=I_1\cap I_2$ which leads to the following exact sequence of $R-$modules:

$$(*) \mbox{ } 0\rightarrow R/I\rightarrow R/I_1\oplus R/I_2\rightarrow R/(I_1+I_2)\rightarrow 0.$$

We have that $$I_1+I_2=\langle x_0,J, \prod_{i=1}^{\beta}(u_ix_0-x_1),x_2-x_0,\ldots,x_n-x_0\rangle$$ $$=\langle x_0, x_1^t,x_2,\ldots,x_n\rangle,$$ where $t=\min\{v,\beta\}$.

With this, $I_1+I_2$ is a complete intersection of codimension $n+1$ and $R/(I_1+I_2)$ has minimal free resolution $$\mathbb E_{*}: 0\rightarrow E_{n+1}=R(-(t+n))\rightarrow\cdots\rightarrow E_1=R(-t)\oplus R(-1)^n\rightarrow R.$$

Also $I_2$ is a complete intersection of codimension $n$ and $R/I_2$ has minimal free resolution $$\mathbb H_{*}: 0\rightarrow H_n=R(-(\beta+n-1))\rightarrow\cdots\rightarrow H_1=R(-\beta)\oplus R(-1)^{n-1}\rightarrow R.$$

Suppose the minimal free resolution of $R/I$ is $$\mathbb F_{*}: 0\rightarrow F_n\rightarrow\cdots\rightarrow F_1\rightarrow R\rightarrow R/I\rightarrow 0.$$

The mapping cone construction (see \cite{e} for background on resolutions) applied to the exact sequence $(*)$ above gives the following free resolution (not necessarily minimal) for $R/(I_1+I_2)$: $$\mathbb W_{*}: 0\rightarrow W_{n+1}=F_n\rightarrow W_n=F_{n-1}\oplus(G_n\oplus H_n)\rightarrow\cdots$$ $$\rightarrow W_1=R\oplus G_1\oplus H_1\rightarrow R^2\rightarrow R/(I_1+I_2)\rightarrow 0.$$

Comparing this with the minimal free resolution we obtained before we get that $E_{n+1}=R(-(t+n))$ is a direct summand of $W_{n+1}=F_n$. So $t+n\geq a(\Gamma)$ and hence, $$t\geq A_n.$$ This leads to the following restriction:

\begin{lem}\label{lem:lemma1} If $s\leq \beta-2$, then $$A_n<\beta.$$
\end{lem}
\begin{proof} If $s\leq\beta-2$, then $t=\min\{v,\beta\}<\beta.$ \end{proof}

\vskip .1in

$\mathbb W_{*}$ is a free resolution of $R/(I_1+I_2)$ and $\mathbb E_{*}$ is a minimal free resolution of the same $R-$module $R/(I_1+I_2)$. From the definition of minimality, one can obtain $\mathbb E_{*}$ from $\mathbb W_{*}$ by removing the redundancies in $\mathbb W_{*}$; that is, some differential maps in $\mathbb W_{*}$ have pieces of degree 0 that can be erased. This process of removing the redundancies will be called a \textit{cancellation}. For example, in the differential $$W_1=R\oplus G_1\oplus H_1\rightarrow R^2,$$ we have the redundancy $R\rightarrow R$ that can be removed to obtain $$G_1\oplus H_1\rightarrow R.$$

\vskip .1in

\begin{lem} \label{lem:lemma2} If $s\geq\beta$, then $$A_n=\beta\mbox{ or }A_n=\beta-1.$$
\end{lem}
\begin{proof} If $s\geq\beta$, then since $v=s$ or $s+1$ we have that $t=\min\{v,\beta\}=\beta$. We saw right before Lemma \ref{lem:lemma1} that $$A_n\leq t=\beta$$ and $$W_{n+1}=F_n=R(-(\beta+n))\oplus K.$$

The only way one has a cancellation in $W_{n+1}$ to obtain $E_{n+1}=R(-(\beta+n))$ is only if $K$ is a direct summand in $$W_n=F_{n-1}\oplus(G_n\oplus H_n).$$ But $K$ is a direct summand in $F_n$ and $0\rightarrow F_n\rightarrow F_{n-1}$ is a part of a minimal free resolution, so there are no cancellations possible here. Therefore, $K$ is a direct summand in $$G_n\oplus H_n=R^k(-(s+n))\oplus R(-(\beta+n-1)).$$

If $A_n\neq\beta$ then $A_n<\beta$ and so $a(\Gamma)=A_n+n<\beta+n$. So $R(-a(\Gamma))$, which is a direct summand in $F_n$, should occur as a direct summand in $K$. So $a(\Gamma)=s+n$ or $a(\Gamma)=n+\beta-1$. Since $s\geq\beta$ we have $a(\Gamma)<\beta+n\leq s+n$ and we are left with $$A_n=\beta-1.$$
\end{proof}

\begin{lem} \label{lem:lemma3} If $s\geq\beta+2$ then $$A_n=\beta.$$
\end{lem}
\begin{proof} We have $s\geq \beta+2$. Again $t=\beta$ and let's assume that $A_n=\beta-1$. From the proof of Lemma \ref{lem:lemma2}, since $A_n=\beta-1$ and therefore $a(\Gamma)=\beta+n-1$, we have that $$K=R^p(-(s+n))\oplus R(-(\beta+n-1)),$$ for some $p\leq k$. So we have $$F_n=R^p(-(s+n))\oplus R(-(\beta+n))\oplus R(-(\beta+n-1)).$$ We must mention that we used the one copy of $R(-(\beta+n-1))$ to obtain the corresponding cancellation in $W_{n+1}$ that gave us $E_{n+1}=R(-(\beta+n))$.

To obtain $E_n=R^n(-(\beta+n-1))\oplus R(-n)$ from $W_n=F_{n-1}\oplus R^k(-(s+n))\oplus R(-(\beta+n-1))$ through a cancellation, since we already used $R(-(\beta+n-1))$ and since $s\geq\beta+2$, then the whole block $R^n(-(\beta+n-1))\oplus R(-n)$ should be a direct summand inside $F_{n-1}$.

We have that $$\mathcal A=\{x_0(x_2-x_0),\ldots,x_0(x_n-x_0),x_0\prod_{i=1}^{\beta}(u_ix_0-x_1)\}$$ is a subset of the minimal generators of $I$. In fact $$F_1=R^{n-1}(-2)\oplus R(-(\beta+1))\oplus \bigoplus R(-a_{1j}).$$

\vskip .1in

\noindent\textit{Claim:} $\min\{a_{1j}\}\geq s$.

\vskip .1in

\noindent\textit{Proof of Claim:} Let $f\in I=I_1\cap I_2$, with $\deg(f)=b < s$. Since $f\in I_1=\langle x_0,J\rangle$, then we can assume that $f=x_0g+h,g\in R$ and $h\in J\cap\mathbb K[x_1,\ldots,x_n]$ with $\deg(h)=b$. Since $J$ is minimally generated in degree $\geq s$, then $h=0$ and we get that $f\in\langle x_0\rangle$. So $f\in \langle x_0\rangle\cap I_2$ and therefore, after the change of variables $x_0'=x_0,x_1'=x_1,x_2'=x_2-x_0,\ldots,x_n'=x_n-x_0$, we have that $$f=x_0'f_0=x_2'f_2+\cdots+x_n'f_n+(\prod_{i=1}^{\beta}(u_ix_0'-x_1'))f_1,$$ where $f_i\in\mathbb K[x_0',\ldots,x_n']$.

We have that $$ht(\langle x_0',x_2',\ldots,x_n',\prod_{i=1}^{\beta}(u_ix_0'-x_1')\rangle)=ht(\langle x_0',x_2',\ldots,x_n',(x_1')^{\beta})=n+1,$$ so $\{x_0',x_2',\ldots,x_n',\prod_{i=1}^{\beta}(u_ix_0'-x_1')\}$ forms a regular sequence and so $f_0\in\langle x_2',\ldots,x_n',\prod_{i=1}^{\beta}(u_ix_0'-x_1')\rangle$. This implies that $$f=x_0'f_0\in \langle x_0'x_2',\ldots,x_0'x_n',x_0'\prod_{i=1}^{\beta}(u_ix_0'-x_1')\rangle.$$ We just proved that if $f\in I$ of degree $\deg(f)<s$, then $f\in\langle\mathcal A\rangle$. So the Claim is shown.

\vskip .1in

\begin{table}[h]
\begin{tabular}{r|ccccc}
&$0$&$1$&$\cdots$&$n-1$&$n$\\
\hline
total:&1&$b_1$&$\cdots$&$b_{n-1}$&$b_n$\\
$0$:&$1$&-&$\cdots$&-&-\\
$1$:&-&$n-1$&$\cdots$&$1$&-\\
$\vdots$&$\vdots$&$\vdots$& &$\vdots$&$\vdots$ \\
$\beta$:&-&$1$&$\cdots$&$n-1$&$1$ \\
$\vdots$&$\vdots$&$\vdots$& &$\vdots$&$\vdots$ \\
\hline
$s-1$:&-&$c_1$&$\cdots$&$c_{n-1}$&$c_n$\\
\hline
$\vdots$&$\vdots$&$\vdots$& &$\vdots$&$\vdots$
\end{tabular}
\end{table}

The table above describes how the betti diagram of $R/I$ should look like. It is important to mention that since $s\geq \beta+2$, then all the syzygies of any order involving at least one minimal generator of $I$ of degree $\geq s$ should occur in the row labeled $s-1$ or below. With this in mind, $R^n(-(\beta+n-1))\oplus R(-n)$ inside $F_{n-1}$ can be obtained only from the (Koszul) syzygies on the set $\mathcal A$. But the $(n-1)-$syzygy module of $\mathcal A$ is $$R^{n-1}(-(\beta+n-1))\oplus R(-n).$$ So if $A_n=\beta-1$, we get an extra $R(-(\beta+n-1))$ in $F_{n-1}$. Contradiction. Consequently, we must have $A_n=\beta$.
\end{proof}

If we put everything together we have:

\begin{thm} Let $\Gamma_1\subset \mathbb P^n$ be a generic set of $\alpha$ points in a hyperplane in $\mathbb P^n$ and let $\Gamma_2\subset \mathbb P^n$ be a set of $\beta$ distinct points on a line in $\mathbb P^n$ not contained in this hyperplane. Suppose that $\alpha\geq \beta+n-2$. Let $\Gamma=\Gamma_1\cup\Gamma_2$. Then:
\begin{enumerate}
  \item If $\alpha<{{\beta+n-3}\choose{n-1}}$, then $d(\Gamma)>A_n$.
  \item If $\alpha\geq{{\beta+n}\choose{n-1}}$, then $d(\Gamma)=A_n$.
\end{enumerate}
\end{thm}
\begin{proof} Since $s$ is the smallest integer such that $\alpha<{{s+n-1}\choose{n-1}}$, then $\alpha<{{\beta+n-3}\choose{n-1}}$ will give us that $s\leq \beta-2$. Similarly, $\alpha\geq{{\beta+n}\choose{n-1}}$ implies that $s>\beta+1$. We obtain the theorem by using Lemma \ref{lem:lemma1} and Lemma \ref{lem:lemma3} above.
\end{proof}

We end with some examples describing what can happen if $s$ is in the range not covered by the theorem above: $s=\beta-1,\beta,\beta+1$. Keeping in mind that $d(\Gamma)=\beta$, we want to see if $d(\Gamma)=A_n$ or not.

\begin{exm}\label{exm:example4} If $s=\beta$, then both situations in Lemma \ref{lem:lemma2} can occur.

First, Example \ref{exm:examplethree} belongs to this situation: $\alpha=n<{{2+n-1}\choose{n-1}}$ (so $s=2$) and $\beta=2$. For this example we have that $d(\Gamma)=A_n$.

Next, consider the following set of $\alpha=6$ points contained in the hyperplane of $\mathbb P^3$ of equation $x_0=0$: $$\Gamma_1=\{[0,0,0,1],[0,1,0,1],[0,0,1,1],[0,1,1,1],[0,2,1,2],[0,-1,-2,1]\}.$$ Disregarding the first coordinate $x_0=0$, we have a set of $6={{2+2}\choose{2}}$ points in $\mathbb P^2$, and so $s=3$. We have that the ideal $J\subset\mathbb K[x_1,x_2,x_3]$ of these points is minimally generated by four cubic generators. So these six points form a generic set of points in $\mathbb P^2$.

Consider the following set of $\beta=3=s$ points on a line in $\mathbb P^3$: $$\Gamma_2=\{[1,7,5,0],[1,3,4,0],[2,10,9,0]\}.$$

Let $\Gamma=\Gamma_1\cup\Gamma_2$ and let $I\subset R=\mathbb K[x_0,x_1,x_2,x_3]$ be the ideal of $\Gamma$. With Macaulay 2 we can obtain the graded minimal free resolution of $R/I$: $$0\rightarrow R(-6)\oplus R(-5)\rightarrow R^6(-4)\oplus R(-3)\rightarrow R^4(-3)\oplus R^2(-2)\rightarrow R.$$

We have $A_3=5-3=2$, and therefore $d(\Gamma)=A_3+1$.
\end{exm}

\vskip .1in

\begin{exm} In the previous example if we remove the last point from the set $\Gamma_1$, we are in the situation of a generic set of five points in the hyperplane $x_0=0$ in $\mathbb P^3$, with $s=2$. Keeping the same $\Gamma_2$ as above (and so $s=\beta-1$), we obtain that $d(\Gamma)=A_3+1$.

If in Example \ref{exm:example4} we keep $\Gamma_1$ as is, and if we remove one point from $\Gamma_2$, we will be in the situation when $s=\beta+1$. With Macaulay 2 we obtain that $d(\Gamma)=A_3$.
\end{exm}

\vskip .2in

\noindent\textbf{Acknowledgements:} We would like to express our sincere gratitude to Prof. Juan Migliore. Without his vision, Section 3 wouldn't have been possible. Also we are very grateful to the anonymous referee for the important corrections, suggestions and comments that made the reading more clear and the statements more precise. Because of his/her questions we were able to improve considerably the lower bound in Lemma \ref{lem:lemma3}.

\renewcommand{\baselinestretch}{1.0}
\small\normalsize 

\bibliographystyle{amsalpha}

\begin{thebibliography}{10}
\bibitem{abm} S. Abrescia, L. Bazzotti, L. Marino,
            {\em Conductor degree and Socle Degree},
            Matematiche (Catania) \textbf{56}(2001), 129-148.

\bibitem{b} L. Bazzotti,
            {\em Sets of points and their conductor},
            J. of Algebra \textbf{283}(2005), 799-820.

\bibitem{e} D. Eisenbud,
            The Geometry of Syzygies, Springer, New York 2005.

\bibitem{gmr} A.V. Geramita, P. Maroscia, L. Roberts,
            {\em The Hilbert function of a reduced k-algebra},
            J. Lond. Math. Soc. \textbf{28}(1983), 443-452.

\bibitem{go} A.V. Geramita, F. Orecchia,
            {\em On the Cohen-Macaulay Type of $s-$lines in $\mathbb A^{n+1}$},
            J. Algebra \textbf{70}(1981),116-140.

\bibitem{gls} L. Gold, J. Little, H. Schenck,
            {\em Cayley-Bacharach and evaluation codes on complete intersections},
            J. Pure Appl. Algebra \textbf{196}(2005), 91-99.

\bibitem{h} J. Hansen,
            {\em Points in uniform position and maximum distance separable codes},
            in: Zero-Dimensional Schemes (Ravello, 1992), de Gruyter, Berlin 1994, pp. 205-211.

\bibitem{k} M. Kreuzer,
            {\em Some applications of the canonical module of a 0-dimensional scheme},
            in: Zero-Dimensional Schemes (Ravello, 1992), de Gruyter, Berlin 1994, pp. 243-252.

\bibitem{ku} A. Kustin, B. Ulrich,
            {\em If the socle fits},
            J. Algebra \textbf{147}(1992), 63-80.

\bibitem{m} J. Migliore,
            {\em Email correspondence}, July-August 2010.

\bibitem{mp} J. Migliore, C. Peterson,
            {\em A symbolic test for $(i,j)-$uniformity in reduced zero-dimensional schemes},
            J. Symbolic Computation \textbf{37}(2004), 403-413.

\bibitem{t1} S. Tohaneanu,
            {\em Lower bounds on minimal distance of evaluation codes},
            Appl. Algebra Eng. Commun. Comput. \textbf{20}(2009), 351-360.

\bibitem{t2} S. Tohaneanu,
            {\em On the De Boer-Pellikaan method for computing minimum distance},
            J. Symbolic Computation \textbf{45}(2010), 965-974.

\bibitem{tvn} M. Tsfasman, S. Vladut, D. Nogin,
            Algebraic Geometric Codes: Basic Notions, AMS, USA 2007.

\end{thebibliography}

\end{document}